\documentclass[11pt,reqno]{article}
\usepackage{amssymb,amsmath,amsthm}
\usepackage[utf8]{inputenc}
\usepackage{tikz}
\usepackage{hyperref}
\usepackage{xcolor}
\usepackage[backend=biber,maxbibnames=99]{biblatex}
\usepackage{authblk} 
\usepackage[symbol]{footmisc}

\usepackage[shortlabels]{enumitem}
\setlist[enumerate]{topsep=0pt,itemsep=-1ex,partopsep=1ex,parsep=1ex}
\usepackage{graphicx}

\newtheorem{statement}{}[section]
\newtheorem{theorem}[statement]{Theorem}

\newtheorem{proposition}[statement]{Proposition}
\newtheorem{definition}[statement]{Definition}
\newtheorem{corollary}[statement]{Corollary}
\newtheorem{remark}[statement]{Remark}

\makeatletter
\newcommand{\subjclass}[2][1991]{%
  \let\@oldtitle\@title%
  \gdef\@title{\@oldtitle\footnotetext{#1 \textbf{Mathematics subject classification:} #2}}%
}
\newcommand{\keywords}[1]{%
  \let\@@oldtitle\@title%
  \gdef\@title{\@@oldtitle\footnotetext{\textbf{Key words and phrases:} #1}}%
}
\makeatother

\def\cantor{\mathfrak{C}}

\def\orb{{\rm Orb}}
\usetikzlibrary{matrix,arrows,shapes.geometric}

\def\dhc{d\text{-}\mathcal{HC}}
\def\dret{d\text{-}\ret}

\def\ret{\mathcal{N}}
\def\id{\text{Id}}

\newcommand{\free}[1]{\mathcal{F}(#1)} 

\def\lip{ \text{Lip}}

\def\NN{\Bbb N}
\def\IN{\hbox{{\rm I}\kern-.13em{\rm N}}}

\def\RR{\mathbb{R}}
\def\IR{\hbox{{\rm I}\kern-.13em{\rm R}}}

\def\sp{\hbox{{\rm span}}}
\def\csp{\overline{\sp}}

\def\aa{\alpha}

\def\dd{\delta}

\def\eor{\hfill{\circledR}}

\def\la{\langle}
\def\ra{\rangle}

\def\ol{\overline}

\title{On disjoint dynamical properties and Lipschitz-free spaces}

\subjclass[\textbf{2020}]{47A16, 
47B37, 
37B02, 
46B20, 
26A18, 
}

\author{Ch. Cobollo\footnote{Christian Cobollo. 
Email: chcogo@upv.es. Corresponding author.
}
and A. Peris\footnote{Alfred Peris. 
Email: aperis@mat.upv.es.
}
}
\affil{\textit{Institut Universitari de Matem\`atica Pura i Aplicada. Universitat Polit\`ecnica de Val\`encia (Spain).}}

\date{}

\begin{document}
\maketitle
\begin{abstract}
The notion of disjoint $\mathcal{A}$-transitivity for a Furstenberg family $\mathcal{A}$ is introduced with the aim to generalize properties derived from disjoint hypercyclic operators. We begin a systematic study by showing some of the basic properties, including necessary conditions to inherit the property on the whole space from 
an invariant linearly dense set containing the origin. As a consequence, we continue the study of the link between non-linear and linear dynamics through Lipschitz-free spaces by presenting some necessary conditions to obtain disjoint $\mathcal{A}$-transitivity for families of Lipschitz-free operators on $\free{M}$ expressed in terms of conditions in the underlying metric space $M$.
\end{abstract}

\section{Introduction and preliminaries}

Following the concept of disjoint dynamical system proposed by Harry Furstenberg in its seminal paper \cite{Furst67}, J. B\`es and the second author introduced the notion of disjoint hypercyclicity \cite{BesPeris07} (see also \cite{Ber-Gon07}). The study of this concept was widely developed in different directions, including recent works like \cite{BMPS12,SanShk14,Bay23,MaMePu22}, to name a few. The objective of this document is double. First, we introduce the notion of being disjoint $\mathcal{A}$-transitive for a given Furstenberg family $\mathcal{A}$.  Roughly speaking, the rôle of the family $\mathcal{A}$ is to specify the ``frequency'' of travelling between open sets for the dynamics of operators satisfying the given property. This allows us to obtain several well-known properties stronger than disjoint transitivity (like disjoint weakly mixing, disjoint mixing, etc.) as particular cases. 

On the other hand, there is a recent line of research devoted to study whether the functor that relates the Lipschitz map $f$ on a given metric space $M$ to its corresponding linearization operator $T_f$ on the Lipschitz-free space $\free{M}$ (see Theorem \ref{thm:linearization}) carries information about a given property. For instance, compactness on Lipschitz-free operators $T_f$ has been studied in \cite{CaJi16} 
 and \cite{JiSeVill14}, the injectivity of $T_f$ in \cite{GarPetPro23}, and the spectrum of weighted Lipschitz-free operators in \cite{ACP2023}. Concerning dynamical properties,  the first reference on this topic was  \cite{MuPe2015}, where M. Murillo-Arcila and the second author proved that being weakly mixing, weakly mixing and chaotic, or mixing, are properties that are inherited from a Lipschitz map $f$ on $M$ to its corresponding operator $T_f$ on $\free{M}$. The study was recently extended by A. Abbar, C. Coine and C. Petitjean in \cite{ACP24,ACP2021}. Also, notions related to recurrence and rigidity for $T_f$ operators were studied in the very recent paper \cite{Ta-Gar24}, where it is proved that no operator $T_f$ can be wild.
 
 In the present work, we are able to tackle the inheritance of disjoint dynamical properties from a given tuple of maps $f_1,\dots,f_N$ on a metric space $M$ to the associated tuple of operators $T_{f_1},\dots ,T_{f_N}$ on  $\free{M}$, and present some necessary conditions to find disjoint $\mathcal{A}$-transitivity for Lipschitz-free operators. Actually, most of the results are stated in a more general framework, namely extending the dynamics of the operators on a linearly dense subset $Z$ containing zero of a general topological vector space $X$, to the whole space $X$. This is not artificially general since, for instance, the logistic map on the unit interval naturally extends to an operator on a non-metrizable topological vector space \cite{MuPe2015}. Also, the interest on the dynamics of operators on non-metrizable topological vector spaces has increased in recent years \cite{Bo00,BoDo12,BoFrPeWe2005,BoKalPe2021,BCarVFav20,DoKa18,GEPe10,Peris18,Shkarin2012}. 

 The document is organized as follows: In Section \ref{sec:disjoint-A-trans}, we introduce disjoint $\mathcal{A}$-transitivity for a Furstenberg family $\mathcal{A}$ and for a given tuple of continuous maps $f_1,\dots,f_N$ on a topological space $Z$. In the same spirit of \cite{MuPe2015}, we provide conditions to inherit disjoint $\mathcal{A}$-transitivity from a tuple of operators restricted to an invariant linearly dense sets containing the origin to the whole space (\ref{thm:d-a-transitive-linear}). Section \ref{sec:disjoint weakly mixing} is fully devoted to the study of the disjoint weakly mixing property. We illustrate the differences between this property and the weakly mixing property for a single operator. In particular, we observe that commuting with a weakly mixing map already provides strong enough disjoint $\mathcal{A}$-transitivity (Theorem \ref{thm:comm-filter} and Proposition \ref{prop:comm-trans}). In Section \ref{sec:free-prel}, we give a brief summary and working tools on Lipschtiz-free spaces, and the final Section \ref{sec:disjoint-free} contains the specific results regarding the $\mathcal{A}$-transitivity for a finite collection of Lipschitz-free operators in $\free{M}$ (Corollary \ref{cor:lip-d-hyp} and Theorem \ref{thm:lip-d-hcc}), illustrated with some examples developed through Subsections \ref{subsec:shift-cantor} and \ref{subsec:anti-sym-tent-map}.

Through this document, unless another thing is specified, $M$ will denote a metric space, $Z$ a Hausdorff topological space,  $X$ a (Hausdorff) topological vector space and $T\colon X\to X$ a 
continuous linear operator (in short, $T\in \mathcal{L}(X)$). Given $N\in \NN$, we will write $Z^N$ (resp. $X^N$) to denote the product space 
$Z\times Z \times ...\times Z$ (resp. $X\oplus X \oplus...\oplus X$) endowed with its natural product 
topology. Given $f_1,...,f_N$ (resp. $T_1,...,T_N$) an $N$-tuple of continuous maps (resp. operators) 
acting on $Z$ (resp. $X$), we will use $f:=f_1\times...\times f_N$ (resp. $T:=T_1\times ...\times T_N$) 
to denote the product map (resp. operator) acting on $Z^N$ (resp. $X^N$).

Given $x\in Z$, we write $x^N=(x,x,...,x)\in Z^N$. Indeed, if we take $A\subset Z$ we can introduce its diagonal set in $Z^N$ as

\begin{equation*}
    \Delta^N(A) :=\{x^N \in Z^N : x\in A\}.
\end{equation*}

We will use the analogous notation in the context of topological vector spaces $X$. We refer to  \cite{BaMa09} and \cite{GEPe10} for standard notation and preliminaries in Linear Dynamics.

\section{Disjoint $\mathcal{A}$-transitiveness}\label{sec:disjoint-A-trans}

In this section, we will introduce the main property of the document and its elementary properties. First, allow us to start by recalling some elementary definitions.

\begin{definition}[\cite{BesPeris07}]\label{def:d-hyper-f}
Let $f_1,...,f_N$ be an $N$-tuple of continuous maps acting on topological space $Z$. By setting $f:= f_1 \times f_2 \times ...\times f_N$, we say that the tuple $f_1,...,f_N$ is \textbf{disjoint  hypercyclic} if there exists $x\in Z$ such that:

\begin{equation*}
    \ol{\orb(x^N, f)}= Z^N.
\end{equation*}

We call such a $x\in Z$ a disjoint hypercyclic point associated to the maps $f_1,...,f_N$ (or just $x\in \dhc(f)$). 


\end{definition}

\begin{definition}[\cite{BesPeris07}]
Let $f_1,...,f_N$ be an $N$-tuple of continuous maps defined in $Z$, and set $f:= f_1 \times f_2 \times ...\times f_N$. It is said that $f_1,...,f_N$ is \textbf{disjoint transitive} if for every non-empty open subsets $U_0, U_1, ..., U_N \subset Z$, setting ${U:=U_1 \times...\times U_N\subset Z^N}$, the \textit{disjoint return set}
 
 \[\dret_f(U_0, U):=\{ m \in \NN : U_0\cap \bigcap_{i=1}^N f_{i}^{-m}(U_i)\neq \emptyset\}\]

is non-empty.
\end{definition}

Notice that an alternative way of thinking about disjoint return sets is as
\begin{equation*}\begin{split}
\dret_f(U_0, U) = \{m\in \NN : \text{ exist } z_m\in U_0 \text{ s.t. } f_i^{m}z_m\in U_i \text{ for } i=1,...,N \}
\end{split}
\end{equation*}

In some literature, the term \textit{diagonally} transitive is used instead of \textit{disjoint}. Through the usage of diagonal sets introduced above, we have the relation
\begin{equation*}
\begin{split}
    \dret_f(U_0, U) &= \{m\in \NN :  f^{-m}(U)\cap U_0^N  \cap \Delta^N(M) \neq \emptyset\}\\
    &= \{m\in \NN :  f^{-m}(U)\cap  \Delta^N(U_0) \neq \emptyset\}\\
    &= N_{f^{-1}}(U,\Delta^N(U_0)), 
\end{split}
\end{equation*}

where $N_g(U,V):=\{m\in \NN : g(U)\cap V \neq \emptyset\}$ denotes a usual return set for a map $g$ and open sets $U$ and $V$.

For instance, thanks to the last equality, the term ``diagonally'' becomes almost self-explanatory.

\begin{remark}
\rm
At this point, it may be worth pointing out the matter of the order when writing this notation. We recall that the disjoint dynamical properties are introduced as a property for a finite family of maps, so the order in which we put the maps is irrelevant to the property. Certainly, for some arguments on the product space $Z^N$, or for the use the notation $\dret_f(U_0, U)$, it is useful to set up a product map of the form $f:= f_1 \times f_2 \times ...\times f_N$. Nevertheless, the order does not matter as long as we make the same rearrangement in the order of the open sets $U_1,...,U_N$ when we write the sets. So, by taking $p:\{1,...,N\}\to \{1,...,N\} $ any permutation of the first $N$ natural numbers, and $g_p:=f_{p(1)} \times f_{p(2)} \times ...\times f_{p(N)}$ we have 
\[\dret_{f}(U_0,U_1\times U_2\times...\times U_N)=   \dret_{g_p}(U_0,U_{p(1)}\times U_{p(2)}\times...\times U_{p(N)})\]\eor
\end{remark}

The following elementary result shows that, by just asking that the disjoint return sets are non-empty, they are indeed infinite. Notice that we use the hypothesis of $Z$ having no isolated points for this. 

\begin{proposition}\label{prop:ret-set-inf}
Let $Z$ be a topological space with no isolated points and let $f_1,...,f_N$ be continuous maps on $Z$. Then, every disjoint return set  $\dret_f(U_0, U)$ is non-empty if and only if every disjoint return set $\dret_f(U_0, U)$ is infinite.
\end{proposition}
\begin{proof}
The ``if'' part is obvious. In order to prove the ``only if'' part, assume that there exists $m\in \NN$ such that the open set

\begin{equation*}
    V:= U_0\cap f_{1}^{-m}(U_1) \cap... \cap f_{N}^{-m}(U_N)
\end{equation*}

is non-empty. Since $Z$ is Hausdorff and has no isolated points, 
we can find disjoint non-empty open sets  $W_1,...,W_N\subset V$. Put $W:=W_1\times...\times W_N$. By hypothesis, we know that there exist $n\in \dret_f(V,W)$, i.e., there exists $z\in V$ such that $f_i^nz \in W_i$ for $i=1,...,N$. Moreover, since the $W_i$'s are disjoint, we know that $n>0$. 

Then, notice that $z\in V \subset U_0$ and also, as $f_i^nz \in W_i\subset V$ for $i=1,...,N$, we deduce that $f_i^{m+n}z= f_i^m(f_i^nz)\in U_i$ for $i=1,...,N$. Thus, we conclude that

\begin{equation*}
    z\in U_0 \cap f_1^{-(m+n)}(U_1)\cap ...\cap f_N^{-(m+n)}(U_N)
\end{equation*}

or, equivalently, $m+n\in \dret_f(U_0,U)$.
\end{proof}

\begin{definition}
 A family $\mathcal{A}\subset \mathcal{P}(\NN)$ is called a Furstenberg family if it satifies:
 \begin{enumerate}
     \item[\rm (i)] Every set $A\in \mathcal{A}$ is infinite;
     \item[\rm (ii)] $\mathcal{A}$ is closed for the inclusion, i.e., if $A\in \mathcal{A}$ and a set $B$ satisfies $A\subset B \subset \NN$ then $B\in \mathcal{A}$.
 \end{enumerate}
 
 If, moreover, it satisfies the additional condition
 
 \begin{itemize}
     \item[\rm (iii)] If $A,B\in \mathcal{A}$ then $A\cap B\in \mathcal{A}$,
 \end{itemize}
 
 then we will say that $\mathcal{A}$ is a filter.
\end{definition}

For any given map $f\colon Z\to Z$, we write  
\begin{equation*}
    \dret_f:=\{A\subset \NN : \dret_f(U_0,U)\subset A \text{ for some open sets } U_0\subset Z, U\subset Z^N\},
\end{equation*}
i.e., the family of supersets of the disjoint return sets.

By Proposition \ref{prop:ret-set-inf}, being disjoint transitive is equivalent to the family $  d\text{-}\ret_f$ being a Furstenberg family. This notion allows us to work in a more general setup:

\begin{definition}\label{def:d-A-trans}
Let $f_1,...,f_N$ be a tuple of continuous maps defined on $Z$, set $f:= f_1 \times f_2 \times ...\times f_N$, and let $\mathcal{A}\subset \mathcal{P(\NN)}$ be a Furstenberg family. We say that $f_1,...,f_N$ is \textbf{disjoint} \textbf{$\boldsymbol{\mathcal{A}}$-transitive} if for every non-empty open subsets $U_0, U_1, ..., U_N \subset Z$,  the disjoint return sets
\begin{equation*}
\dret_f(U_0, U)\in  \mathcal{A}, 
\end{equation*}
where $U:=U_1 \times...\times U_N\subset Z^N$. 
\end{definition}

The following is the main result of this section.

\begin{theorem} \label{thm:d-a-transitive-linear}
Let $X$ be a topological vector space, let $T_1,...,T_N$ be an $N$-tuple of operators in $\mathcal{L}(X)$, and let $0\in Z\subset X$ be a linearly dense set which is invariant for each $T_i$. Assume that the tuple $T_{1|Z},...,T_{N|Z}$ is disjoint $\mathcal{A}$-transitive  in $Z^N$, with $\mathcal{A}$ being a filter. Then $T_1,...,T_N$ are disjoint $\mathcal{A}$-transitive.
\end{theorem}
\begin{proof}
Take $V_0,V_1,...,V_N\subset X$ non-empty open sets. Put $Y:= \sp(Z)$ (as $Y$ is dense in $X$, it is enough to prove that $T_{1|Y},...,T_{N|Y}$ are disjoint topologically transitive in $Y^N$). Without loss of generality, we can find $z_0,z_1,..., z_N$, such that
\begin{equation*}
    z_i= \sum_{k=1}^{k_0} \aa_{i,k} x_{i,k} \in V_i
\end{equation*}
 where $x_{i,k}\in Z$, for each $k\in \{1,...,k_0\}$ and  $i\in \{0,1,...,N\}$. Now, for each $x_{i,k}$ we can find an relatively open neighbourhood $U_{i,k}\subset Z$ such that

\begin{equation*}
  z_i\in   \sum_{k=1}^{k_0} \aa_{i,k} U_{i,k} \subset V_i
\end{equation*}

for $i\in \{0,1,...,N\}$.

Now, as $0\in Z$,  for each $i\in \{0,1,...,N\}$ we can find $N \times k_0$ relatively open neighbourhoods of the origin $W_{j,k}^i\subset Z$ (for every $j=\{0,1,...,N\} $ with $j\neq i$ and every $k\in \{1,...,k_0\}$) so small that the following condition holds:

\begin{equation}\label{eq:incl}
    \sum_{k=1}^{k_0} \aa_{i,k} U_{i,k} +
     \sum_{\substack{j=0 \\ j\neq i}}^N\sum_{k=1}^{k_0} \aa_{j,k} W_{j,k}^i \subset V_i.
\end{equation}

Thus, in total, as there are $N+1$ rows --one per each open set $V_i$-- we just found $(N+1)\times N \times k_0$ very small open neighbourhoods of the origin of the form $W_{i,k}^j$. Notice that  we can even take all equal $W_{i,k}^j:=W$, as long as we pick a small enough $W$, which allows us to stress how this notation works in order to make things easier for the reader.
In each of the $N+1$ rows, all the coefficients appear. In the $i$-th row, for each $k\in \{1,...,k_0\}$ the coefficient $\aa_{i,k}$ goes with the open neighbourhood $U_{i,k}$, and for the rest of the coefficients, every $\aa_{j,k}$ where $j\neq i$ goes with the open neighbourhood $W_{j,k}^i$ ($W_{j,k}^i$ is the open neighbourhood that goes with the coefficient $\aa_{j,k}$ in the $i$-th row).

Thus, for each $i\in \{0,1,...,N\}$ and $k\in \{1,..., k_0\}$ we have a coefficient $\aa_{i,k}$. This coefficient has $N+1$ associated open neighbourhoods, the $U_{i,k}$ and the other $W_{i,k}^j$ with $j\in\{0,1,...,N\}$ with $j\neq i$. Moreover, they are placed in this order: $W_{i,k}^0, W_{i,k}^1,...,W_{i,k}^{i-1}, U_i,W_{i,k}^{i+1},...,W_{i,k}^{N} $.

Thus, put the sets
\begin{equation*}
 \ret_{i,k}:=   \dret_{ T_{|M^N} }\big(W_{i,k}^0 , W_{i,k}^1\times...\times W_{i,k}^{i-1}\times U_i\times W_{i,k}^{i+1}\times ...\times W_{i,k}^{N}\big)\in  \mathcal{A}.
\end{equation*}

By linearity of the operators and the inclusions in equation \eqref{eq:incl},
we deduce

\begin{equation*}
    \dret_ {T} \big(V_0, V_1\times  V_2\times ...\times V_N \big) \supset  \bigcap_{\substack{i\in\{0,1,...,N\} \\ 
    k\in \{1,...,k_0\} }}  \ret_{i,k} \in \mathcal{A},
\end{equation*}

as we wanted to prove.

\end{proof}


\begin{remark} \label{rem:mixing}
\rm The result above depends on being disjoint $\mathcal{A}$-transitive where $\mathcal{A}$ is a filter, and thus it is important to know, at least, that some disjoint property fulfils  this condition. Notice that if the maps $f_1,...,f_N$ are disjoint mixing, then every element in  $d$-$\ret_f$ is a cofinite set, so $d$-$\ret_f$ is a filter. \eor
\end{remark}

It is worth noticing that the argument used in Theorem \ref{thm:d-a-transitive-linear} for $N=1$ (i.e., a non-disjoint version of the result for a single operator $T$) provides a generalization of \cite[Theorem 2.3]{MuPe2015} for Furstenberg Families.


\section{Disjoint weakly mixing equivalences}
\label{sec:disjoint weakly mixing}

In this section, we want to discuss whether the context of the disjoint weakly mixing property is analogous to the one with the usual weakly mixing property.

\begin{definition}
Let $f_1,...,f_N$ be a family of $N\geq 2$ continuous maps defined on $Z$, and $r\in \NN$. We will say that they are \textbf{disjoint weakly mixing of order $\boldsymbol r$} if the $N$ product maps $\overbrace{f_1\times ...\times f_1}^{r\text{-times}},..., \overbrace{f_N\times...\times  f_N}^{r\text{-times}}$ acting on the product space $Z^r$ are disjoint transitive. If the maps $f_1,...,f_N$ are disjoint weakly mixing of order 2, then we will simply say that they are \textbf{disjoint weakly mixing}.
\end{definition}

\begin{proposition}
\label{prop:d-wm_intersec}
The following statements are equivalent:
\begin{enumerate}
    \item[\rm (i)] The maps $f_1,...,f_N$ are disjoint weakly mixing;
    \item[\rm (ii)] For every $U_0,V_0\subset Z$ non-empty open sets and $U, V\subset Z^N$ with $U:=U_1\times...\times U_N$ and $V:=V_1\times...\times V_N$ product of $N$ non-empty open sets of $Z$, we have
\begin{equation*}
    \dret_f(U_0,U) \cap \dret_f(V_0,V) \neq \emptyset.
\end{equation*}
\end{enumerate}

Moreover, if $Z$ has no isolated points, then both are also equivalent to:
\begin{enumerate}
\item[\rm (iii)] For every $U_0,V_0\subset Z$ non-empty open sets and $U, V\subset Z^N$ with $U:=U_1\times...\times U_N$ and $V:=V_1\times...\times V_N$ product of $N$ non-empty open sets of $Z$, the set $\dret_f(U_0,U) \cap \dret_f(V_0,V)$ is infinite.
\end{enumerate}
\end{proposition}




\begin{proof}
 For every $i=1,...,N$ put $g_i:=f_i\times f_i$ defined in the product space $Z^2$ and  $g:=g_1\times...\times g_N$ a product map on the product space $(Z^2)^N$. By definition, $f_1,..., f_N$ in $Z$ are disjoint weakly mixing transitive if and only if $g_1,...,g_N$ in $Z^2$ are disjoint transitive, i.e., for every non-empty open subsets $W_0,W_1,...,W_N\subset Z^2$, putting $W:=W_1\times...\times W_N\subset (Z^2)^N$, the disjoint return sets $\dret_g(W_0,W)$ are non-empty.

Assume $U_0,V_0,U$ and $V$ as in the statement. Now, we may consider ${W_0:=U_0\times V_0}$ and $W_i:=U_i\times V_i$ for every $i=1,...,N$. {Notice that, as the intersection of products is the product of intersections,
\begin{equation*}
    \begin{split}
        W_0 \cap \bigcap_{i=1}^N g_i^{-m}(W_i)&= \bigg(U_0\times V_0\bigg)\cap \bigcap_{i=1}^N \bigg( (f_i\times f_i)^{-m}\big(U_i\times V_i\big)\bigg)\\
        &=\bigg( U_0\times V_0\bigg) \cap \bigcap_{i=1}^N \bigg( f_i^{-m}\big(U_i\big)\times f_i^{-m} \big( V_i\big)\bigg)\\
        &= \bigg(U_0\cap  \bigcap_{i=1}^N  f_i^{-m}\big(U_i\big)  \bigg) \times \bigg(V_0\cap  \bigcap_{i=1}^N  f_i^{-m}\big(V_i\big)  \bigg)
    \end{split}
\end{equation*}
}

Then, an easy computation shows
\small{
\begin{equation*}
\begin{split}
   \dret_f(U_0,U) \cap \dret_f(V_0,V) 
  &= \dret_f(U_0,U)\bigcap \{m\in \NN: V_0\cap \bigcap_{i=1}^N f_i^{-m}(V_i) \neq \emptyset\}\\
   &= \{m\in\NN : W_0 \cap g_1^{-m}(W_1)\cap...\cap g_N^{-m}(W_N) \neq \emptyset\}\\
    &=  \dret_g(W_0,W).
\end{split}
\end{equation*}

Now, consider $\mathcal{B}:=\{A\times B : A,B\subset Z \text{ non-empty open sets}\}$. Thus, item (ii) is equivalent to $\dret_g(W_0,W_1\times...\times W_N)$ being non-empty for every $W_0,W_1,...,W_N\in \mathcal{B}$. As $\mathcal{B} \cup \{\emptyset\}$ is a base for the product topology in $Z^2$, the equivalence between (I) and (ii) is proven. Moreover, if $Z$ has no isolated points, then the same happens for $Z^2$, so by Proposition \ref{prop:ret-set-inf}, it is equivalent to ask for the sets $\dret_g(W_0,W)$ being non-empty or infinite.
}
\end{proof}

  The same argument would prove that the maps $f_1,...,f_N$ are disjoint weakly mixing of order $r$ if and only if the intersection of $r$ disjoint return sets is non-empty.

\begin{proposition}
\label{prop:d-wm-order-r}
The following statements are equivalent:
\begin{enumerate}
    \item[\rm (i)] The maps $f_1,...,f_N$ are disjoint weakly mixing of order $r$;
    \item[\rm (ii)] For every $U_{0,1},...,U_{0,r}\subset Z$ non-empty open sets and $U_1,...,U_r\subset Z^N$ with $U_j:=U_{j,1}\times...\times U_{j,N}$ for $j=1,...,r$, product of $N$ non-empty open sets of $Z$, we have
\begin{equation*}
   \bigcap_{j=1}^r \dret_f(U_{0,j},U_j) \neq \emptyset.
\end{equation*}
\end{enumerate}

Moreover, if $Z$ has no isolated points, then both are also equivalent to:
\begin{enumerate}
\item[\rm (iii)] For every $U_{0,1},...,U_{0,r}\subset Z$ non-empty open sets and $U_1,...,U_r\subset Z^N$ with $U_j:=U_{j,1}\times...\times U_{j,N}$ for $j=1,...,r$, product of $N$ non-empty open sets of $Z$, the set $ \displaystyle \bigcap_{j=1}^r \dret_f(U_{0,j},U_j)$ is infinite.
\end{enumerate}
\end{proposition}

Given two Furstenberg families $\mathcal{A}_1, \mathcal{A}_2$, we define its product by 

\[\mathcal{A}_1 \cdot \mathcal{A}_2:= \{ A_1\cap A_2: A_1\in \mathcal{A}_1, \  A_2\in \mathcal{A}_2 \}. \]

The equivalence between (i) and (iii) in Proposition \ref{prop:d-wm-order-r}.

\begin{corollary}
    Let $Z$ have no isolated points. Then $f_1,...,f_N$ are disjoint weakly mixing of order $r$ if and only if $\mathcal{A}:= \overbrace{\dret_f \cdot \cdot \cdot \dret_f}^{r \text{-times}} $ is a Furstenberg family. In that case, $f_1,...,f_N$ are disjoint $\mathcal{A}$-transitive.
\end{corollary}

\begin{remark}
\rm In \cite[Theorem 2.7]{BesPeris07} it is proved for operators $T_1,...,T_N$ that satisfying the Disjoint Hypercyclicity Criterion is equivalent to the fact that, for every $r\in \NN$ the operators $\oplus_{i=1}^r T_1,..., \oplus_{i=1}^r T_N$ are disjoint transitive operators in of $\mathcal{L}(X^r)$, {i.e., $T_1,...,T_N$ are disjoint weakly mixing of any order.} 
 
 This \cite[Theorem 2.7]{BesPeris07} is the disjoint analogous to the classical Hypercyclicity Criterion (see \cite{BePe99}),
  which states that satisfying the Hypercyclicity Criterion is equivalent to having the weakly mixing property. This can also be combined with the classical Furstenberg theorem, which states that a map $f$ is weakly mixing if and only if every product $f\times...\times f$ is disjoint transitive, and in particular, the family of supersets of return sets is a filter.
  
  Thus, one may be tempted to think that the Disjoint Hypercyclicity Criterion may also be equivalent to the disjoint weakly mixing property or the filter condition. However, this is not the case. In \cite{SanShk14}, Sanders and Shkarin construct disjoint weakly mixing operators that do not satisfy the Disjoint Hypercyclicity Criterion. Indeed, they proof that in any infinite-dimensional Banach space a family of operators $T_1,...,T_N$ can be constructed such that $\bigoplus^k T_1,...,\bigoplus^k T_N$ are disjoint transitive but $\bigoplus^{k+1} T_1,...,\bigoplus^{k+1} T_N$ are not disjoint hypercyclic (then, they are not disjoint transitive). Thus, the Sanders and Shkarin construction also shows that a disjoint version of the Furstenberg theorem is not true in general, although it can be recovered under some commutativity assumptions---see Proposition \ref{prop:comm-trans} below.  \eor
\end{remark}

\subsection{Commutativity arguments}\label{subsec:commutativity}

Unlike the classical $N=1$ case, it is well known that there are maps $f_1,...,f_N$ that are disjoint hypercyclic but not disjoint transitive---see \cite[Corollary 3.5]{SanShk14}. However, Cardeccia proved that, for the linear case, and under some commutativity assumptions, both notions are equivalent.

\begin{proposition}[{\cite[Proposition 2.4]{Car24}}]
    Let $T_1,...,T_N$ be operators in an $F$-space such that $T_1$ commutes with $T_i$ for $i=2,...,N$. Then, $T_1,...,T_N$ are disjoint transitive if and only if they are disjoint hypercyclic.
\end{proposition}

Also in \cite{SanShk14}, Sanders and Shkarin constructed disjoint weakly mixing operators not satisfying the Disjoint Hypercyclicity Criterion, separating both concepts in the disjoint case. We will show that this construction is impossible under some commutativity assumptions. Notice that, unlike the aforementioned ones, our results do not require linear assumptions on the dynamical system.

\begin{theorem}\label{thm:comm-filter}
     Let $Z$ have no isolated points and let $f_1,...,f_N$ be continuous maps on $Z$ such that there exists a weakly mixing map $g\colon Z\to Z$ such that 
     \[g\circ f_i = f_i \circ g \quad \text{for } i=1,...,N.\]

     Then, $\dret_f$ is a filter.
\end{theorem}

\begin{proof}
    For $i=0,...,N$, let $\emptyset \neq U_i, V_i \subset Z$ be arbitrary open sets. We want to show that for each $i=0,...,N$ there exist $W_i\subset Z$ open sets such that

    \[\dret_f(W_0,W)\subset \dret_f(U_0,U) \cap \dret_f(V_0,V)\]

    where, as usual, 
    \begin{align*}
        U&=U_1\times...\times U_N, \\
         V&=V_1\times...\times V_N, \\
          W&=W_1\times...\times W_N.
    \end{align*}

    By Furstenberg Theorem (see for instance \cite[Theorem 1.51]{GEPe10}) it is known that the $N+1$ product $ \overbrace{g\times...\times g}^{(N+1)\text{-times}}\colon Z^{N+1}\to Z^{N+1} $ is transitive. 
Thus, there exist $m\in \NN$ such that $m\in N_g(U_i,V_i)$ for every $i=0,..,N$. 

Then, for $i=0,...,N$, take $W_i:=g^{-m}(V_i) \cap U_i$, which is a non-empty open set in $Z$. Since $W_i\subset U_i$, it is clear that 

\[\dret_f(W_0,W) \subset \dret_f(U_0,U).\]

Now, consider $m'\in \dret_f(W_0,W)$. This means that there exist

\[w\in W_0 \cap \Big( \bigcap_{i=1}^N f_i^{-m'}(W_i) \Big)
.\]

Notice that $v:= g^m(w)\in g^m(g^{-m}(V_0))\subset V_0$.

We only have to show that $f_i^{m'}(v)\in V_i$ for $i=1,..., N$. Indeed, by using the commutativity of $g$ with $f_i$, 

\[f_i^{m'}(v) = f_i^{m'}(g^m(w))= g^m(f_i^{m'}(w))\in g^m(W_i) \subset V_i.\]

The proof is over.
\end{proof}

This allows us, under an additional commutativity hypothesis, to obtain a disjoint analogous of the Furstenberg Theorem. 
\begin{proposition}\label{prop:comm-trans}
     Let $Z$ have no isolated points and let $f_1,...,f_N$ be disjoint $\mathcal{A}$-transitive maps on $Z$ such that there exists a weakly mixing map $g\colon Z\to Z$ with
     \[g\circ f_i = f_i \circ g \quad \text{for } i=1,...,N.\]
     
     Then, for any $r\in \NN$, the $N$-tuple $ \overbrace{f_1\times...\times f_1}^{r\text{-times}},...,\overbrace{f_N\times...\times f_N}^{r\text{-times}}$ is disjoint $\mathcal{A}$-transitive. In particular, $f_1,...,f_N$ are disjoint weakly mixing of any order.
\end{proposition}

\begin{proof}
    By the previous result $\dret_f$ is a filter, so as a consequence the $N$-tuple of $r$-products $ \overbrace{f_1\times...\times f_1}^{r\text{-times}},...,\overbrace{f_N\times...\times f_N}^{r\text{-times}}$ is again disjoint $\dret_f$-transitive, and we had that $\dret_f\subset \mathcal{A}$ because the original $N$-tuple $f_1,...,f_N$ was disjoint  $\mathcal{A}$-transitive. 
\end{proof}

Given a map $f\colon Z \to Z$, we call the \textbf{commutator} of $f$ to the set
\[ \mathcal{C}_f:=\{g\colon Z\to Z: g\circ f = f\circ g\}. \]

\begin{corollary}
    Let $f_1,...,f_N$ be continuous maps on $Z$ such that they are disjoint weakly mixing of order $r$ but not disjoint weakly mixing of order $r+1$. Thus, there is no weakly mixing map inside $\cap_{i=1}^N \mathcal{C}_{f_i}$.
\end{corollary}

\section{Summary on Lipschitz-free spaces} \label{sec:free-prel}

\subsection{Introduction}

 The origin of this class of Banach spaces---which are also referred to through the literature as Arens--Eells spaces or Transportation Cost spaces---can be attributed to classical authors like Kantorovich and Rubinstein \cite{Kant-Rub} through its works in optimal transportation problems. As these spaces turned out to play a relevant role in different fields, their properties were studied and rediscovered many times from different perspectives by authors like Arens, de Leeuw, Eells, or Johnson---see \cite{Ar-Ee}, \cite{Johnson}, and \cite{deLeeuw}. Meanwhile, early versions of what nowadays is known as the universal property were provided by Kadets \cite{Kadets} and Pestov \cite{Pestov}, as Lipschitz-free Banach space may also be presented as the free object between the categories of metric spaces (with the morphisms of Lipschitz maps) and Banach spaces (with the morphisms of linear bounded operators). In 1999, the first edition of the celebrated monograph by Weaver can be considered the starting line of its systematic study as Banach spaces--focusing also on the Banach algebra structure of their duals--, and its recent second edition \cite{Weaver2} has taken its place as the main reference in the topic. The seminal paper \cite{GoKa03} by Godefroy and Kalton in 2003 started the study of Lipschitz-free spaces from the Geometry of Banach spaces point of view, providing some remarkable results using Lipschitz-free spaces as its key tool, like the preservation of the bounded approximation property of Banach spaces by Lipschitz-homeomorphisms, or the linearization of metric embeddings of a separable Banach space $X$ inside a Banach space $Y$. Nowadays, the geometry of Lipschitz-free spaces remains a very active field of research, as it is still unknown in so many aspects.


\subsection{Lipschitz-free working preliminaries} \label{subsec:free-preliminaries}

Although a proper systematic introduction to these spaces may be found in references like \cite{Weaver2}, here we briefly summarise some of its basic properties and tools needed to develop this work.
  
Let $(M,d)$ be a metric space. By choosing a distinguished point $0_M\in M$ (usually called just $0$ when there is no possibility of confusion), we can take $\lip_0(M)$ as the linear space of Lipschitz functions $f\colon M\to \RR$ that vanish at $0$, which is a real Banach space when endowed with the norm of the Lipschitz constant, $\|\cdot\|_{\text{Lip}}$, defined by
\begin{equation*}
  \|f\|_{\text{Lip}}:= \sup_{x\neq y\in M} \dfrac{|f(x)-f(y)|}{d(x,y)}.
\end{equation*}

The choice of the element $0$ is arbitrary since the resulting spaces of taking different distinguished points are isometrically isomorphic as real Banach spaces. It is also worth noticing that, from the Banach space perspective, it is enough to consider complete metric spaces, as the $\lip_0$ (and therefore, the Lipschitz-free space) is isometrically isomorphic whether you take $M$ or its completion. 

The map $\dd: M \to \lip_0(M)^*$ that sends every point $x\in M$ to its evaluation functional $\dd_x$ is an isometry, and it is easy to see that $\dd_x$ and $\dd_y$ with $x \neq y$ are linearly independent. Notice that $\csp \{ \dd (x) : x\in M \}$ is a closed linear subspace of $\lip_0(M)^*$, and indeed it is a predual of $\lip_0(M)$, denoted by $\free{M}$, i.e.,
\[ \free{M}:= \csp\{\dd_x : x\in M\} \ \ ( \subset \lip_0(M)^* ).\]
  
This is the so-called \textit{Lipschitz-free space} of $M$. Roughly speaking, we can think of $\free{M}$ as the Banach space constructed by taking $M$ and providing it with a linear structure in which distinct (non-zero) points in $M$ are now linearly independent, and the endowed norm is the one keeping the original metric structure of $M$. Thus, $\|\dd_x\|= d(x,0)$, or more generally, $\|\dd_x-\dd_y\|=d(x,y)$.
  
As the linear continuous functionals on $\free{M}$ are the elements of $\lip_0(M)$, it is clear how they act. Let $g\in \lip_0(M)$ and $\sum_i^n a_i \dd_{x_i}\in \free{M}$, then
\[\la g, \sum_i^n a_i \dd_{x_i} \ra= \sum_i^n a_i \la g,   \dd_{x_i} \ra= \sum_i^n a_i \dd_{g(x_i)}.\]
  
 Given $x,y \in M$ two distinct points in $M$, its associated \textbf{molecule} (some authors call this \textit{elemental molecule}) is
 \[m_{x,y}:= \frac{\dd_x -\dd_y}{d(x,y)}\in \free{M}.   \]
  
Although $\free{M}$ space may be fully described using \textit{the Dirac deltas}, considering molecules provides some advantages. As, for every $x\in M$, ${\dd_x}={d(x,0)}\cdot m_{x,0}$, $\free{M}= \csp\{m_{x,y}: x,y \in M\}$. The reader will rapidly understand the relevance of these elements on $\free{M}$ since they are not only norm-one elements but, as $\la g, m_{x,y} \ra= \frac{g(x)-g(y)}{d(x,y)}$, they constitute a $1$-norming subset of $\free{M}$ for the norm $\|\cdot\|_{\lip}$. These elements also play a relevant role in the extreme structure of $\free{M}$--for instance, it is still conjectured that every extreme point in a Lipschitz-free space should be a molecule, and it remains as one of the main open problems in the area.

Given $(M,d)$ and $(N,p)$ two metric spaces with distinguished points $0_M$ and $0_N$, we denote by $\lip_0(M,N)$ the space of Lipschitz maps $f\colon M\to N$ such that $f(0_M)=0_N$.
The following property is the quintessence that permits the relation of non-linear dynamics with linear dynamics using Lipschitz-free spaces.

\begin{theorem}\label{thm:linearization}
    Let $(M,d)$ and $(N,p)$ metric spaces with distinguished points $0_M$ and $0_N$ (respectly), and $f\in \lip_0(M,N)$. Then, there exist a unique linear and bounded operator $T_f\colon \free{M}\to \free{N}$ such that $\|T_f\|=\|f\|_\lip$ and the following diagram commutes:

    \begin{center}
\begin{tikzpicture}
  \matrix (m) [matrix of math nodes,row sep=3em,column sep=4em,minimum width=2em]
  {
     M & N \\
     \free{M} & \free{N} \\};
  \path[-stealth]
    (m-1-1) edge node [left] {$\delta_M$} (m-2-1)
            edge node [above] {$f$} (m-1-2)
    (m-2-1.east|-m-2-2) edge [double] node [below] {$T_f$}
            node [above] {} (m-2-2)
    (m-1-2) edge node [right] {$\delta_N$} (m-2-2) ;
\end{tikzpicture}
\end{center}
\end{theorem}

Thus, every $f\in \lip_0(M,M)$ has an associated operator $T_f\in \mathcal{L}(\free{M})$, and this relation allows us to wonder about what condition is needed to ask on the original dynamical system $(M,f)$ to have a particular dynamical property in $(\free{M},T_f)$.

\section{Inheritance of disjoint dynamical properties in Lipschitz-free spaces} \label{sec:disjoint-free}

For the specific set-up of Lipschitz-free spaces, Theorem \ref{thm:d-a-transitive-linear} translates into the following one that allows the inheritance of a given (strong enough) disjoint dynamical property from Lipschitz maps into its corresponding tuple of Lipschitz-free operators.

\begin{corollary}\label{cor:lip-d-hyp}
Let $f_1,...,f_N\in \lip_0(M,M)$ which are disjoint $\mathcal{A}$-transitive maps in $M$, with $\mathcal{A}$ being a filter. Then its corresponding linear operators $T_{f_1},...,T_{f_N}$ are disjoint $\mathcal{A}$-transitive in $\free{M}$.
\end{corollary}

Another result concerning the specific setup of  Lipschitz-free spaces is the criterion below. It is somehow a Lipschitz translation of the original d-Hypercyclicity Criterion in \cite{BesPeris07}, in the same spirit as the Lipschitz Hypercyclicity Criterion presented in \cite{ACP2021}). 

\begin{theorem}(Lipschitz d-Hypercyclicity Criterion)\label{thm:lip-d-hcc}

Let $f_1,f_2,...,f_N\in \lip_0(M,M)$ and let $(n_k)$ be a strictly increasing sequence of positive integers. We say that the functions satisfy the Lipschitz Disjoint Hypercyclicity Criterion with respect to $(n_k)$ provided there exist dense subsets $M_0,M_1,...,M_N\subset M$ and mappings $g_{i,k}:M_i\to M$ ($1\leq i\leq N$, $k\in \NN$) satisfying

\begin{enumerate}
    \item[\em($i$)] $d\big(f_{i}^{n_k}(x_0),0\big) \xrightarrow{k\to \infty} 0$ for every $x_0\in M_0$.
    \item[\em($ii$)] $d\big(g_{i,k} (x_i),0\big) \xrightarrow{k\to \infty} 0$ for every $x_i\in M_i$.
    \item[\em($iii$)]  $d\big((f_{i}^{n_k} \circ g_{j,k}) (x_i), \Delta_{j,i}x_i \big) \xrightarrow{k\to \infty} 0$ for every $x_i\in M_i$ ($j\in \{1,...,N\}$), where $\Delta_{j,i}x_i$ is $x_i$ if $j=i$, and the distinguished point $0\in M$ otherwise.
\end{enumerate}

If this happens, then the operators $T_{f_1},...,T_{f_N}$ satisfy the Disjoint Hypercyclicity Criterion in \cite{BesPeris07}. Thus, the sequences $\{T_{f_1}^{n_k}\}_{k=1}^\infty,...,\{T_{f_N}^{n_k}\}_{k=1}^\infty$ are disjoint mixing, and in particular, $T_{f_1},...,T_{f_N}$ are disjoint transitive.
\end{theorem}

\begin{proof}
For $i\in \{0,1,...,N\}$ put $X_i:= \sp\{\dd(M_i)\}$, which are all dense sets in $\free{M}$. As, $T_{f_i^{n_k}}= T_{f_i}^{n_k}$, thanks to condition $(i)$ we know that for $i\in \{1,...,N\}$ $T_{f_i}^{n_k} \xrightarrow{k\to \infty} 0$ pointwise on $X_0$.

Now, put $S_{i,k}:X_i\to \free{M}$ for the linear extension of the map 

\begin{equation*}
    \dd \circ g_{i,k}\circ \dd^{-1}_{|\dd(M_i)}: \dd(M_i) \to \free{M}.
\end{equation*}

By item $(ii)$ we deduce $S_{i,k}\xrightarrow{k\to \infty} 0$ pointwise on $X_i$.

Finally, if $z_i=\sum_{n=1}^{n_0} \aa_n \dd(x_{i,n}) \in X_i$ (for $i\in \{1,...,N\}$) then 

\begin{equation*}
\begin{split}
     \|( T_{f_i}^{n_k} S_{j,k} - \Delta_{j,i} \text{Id}) z_i \| & \leq \sum _{n=1}^{n_0} |\aa_n|\| ( T_{f_i}^{n_k} \circ S_{j,k} - \Delta_{j,i} \text{Id})\dd(x_{i,n})\| \\
     & = \sum _{n=1}^{n_0} |\aa_n|\| \dd( T_{f_i}^{n_k}\circ S_{j,k})(x_{i,n}) - \dd(\Delta_{j,i} x_{i,n})\| \\
     & = \sum _{n=1}^{n_0} |\aa_n| d\big((f_{i}^{n_k} \circ g_{j,k}) (x_{i,n}), \Delta_{j,i}x_{i,n}\big),
\end{split}
\end{equation*}

which tends to $0$ as $k\to \infty$, since by item $(iii)$ we have  
\[{d\big((f_{i}^{n_k} \circ g_{j,k}) (x_{i,n}), \Delta_{j,i}x_{i,n}\big)\xrightarrow{k\to \infty} 0}\]

for every $n\in \{1,...,n_0\}$.

Thus, the operators $T_{f_1},...,T_{f_N}\in \mathcal{L}(\free{M})$ satisfy the Disjoint Hypercyclicity Criterion with respect to the sequence $(n_k)$ (see \cite[Proposition 2.6]{BesPeris07}). In particular, the sequences $\{T_{f_1}^{n_k}\}_{k=1}^\infty,...,\{T_{f_N}^{n_k}\}_{k=1}^\infty$ are disjoint mixing, and thus $T_{f_1},...,T_{f_N}$ are disjoint transitive.
\end{proof}

The rest of the document will be devoted to exploring the definition and properties of some natural examples of Lipschitz-free operators to illustrate an application of the previous results.

\subsection{The family of inherited backward shifts on the Lipschitz-free of the Cantor set}\label{subsec:shift-cantor}

Consider $\cantor$ the usual middle third Cantor set with the usual distance. This set may be written as $\cantor=[0,1]\backslash \big( \bigcup_{n=1}^\infty I_n \big)$, where the sequence of intervals $I_n:=]a_n,b_n[$ is an enumeration of the middle thirds we are taking on each step. In particular, we consider $I_1=]\frac{1}{3},\frac{2}{3}[$, $I_2= ]\frac{1}{9},\frac{2}{9}[$, $I_3=]\frac{7}{9}, \frac{8}{9}[$, etc., i.e., we order the removed intervals following Figure \ref{fig:Cantor-order}.


 \begin{figure}[htb!]
\begin{center}
\includegraphics[width=4in]{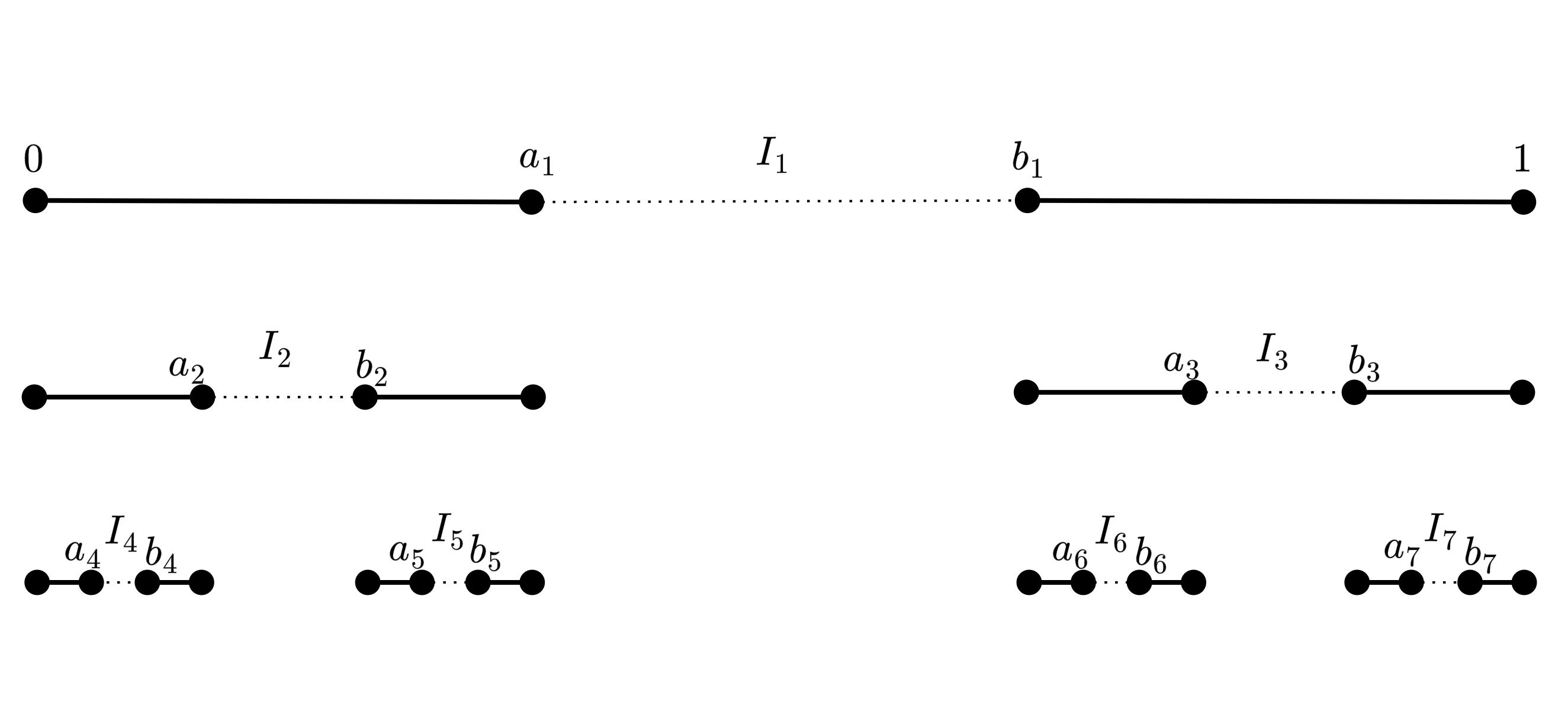}
\caption{Enumeration of the gaps in the ternary Cantor set.}
\label{fig:Cantor-order}
\end{center}
\end{figure}

It would be useful to recall that 
\begin{equation*} \label{eq:Cantor-sequence}
      (d(b_n,a_n))_{n=1}^\infty = \Big(\frac{1}{3},\frac{1}{3^2},\frac{1}{3^2}, \frac{1}{3^3}, \frac{1}{3^3},\frac{1}{3^3},\frac{1}{3^3},... \Big),
\end{equation*}  
    
    i.e., the sequence such that, for $i\geq 1$, the term $\frac{1}{3^i}$ appears $2^{i-1}$ times.

It is well known that every point of the Cantor set has a ternary representation as a sequence of digits $0$ and $2$. Let $t\in \cantor$, therefore $t= \sum \frac{s_n(t)}{3^n}$ where $s_n(t)\in \{0,2\}$. Thus, define the ternary representation of $t\in \cantor$ by $s(t):= (s_1(t), s_2(t),...)$.

Now, consider the (Lipschitz) map $\sigma \colon \cantor \to \cantor$ known as the \textit{backward shift}, defined by the expression 
\[ \sigma(t):= \sum \frac{s_{n+1}(t)}{3^n}. \]
Basically, for a point $t$ with ternary representation $s(t)=(s_1(t),s_2(t), s_3(t),...)$, the image through the backward shift $\sigma(t)$ is the point of the Cantor set with ternary representation $(s_2(t), s_3(t),...)$.

 \begin{figure}[htb!]
\begin{center}
\includegraphics[width=4in]{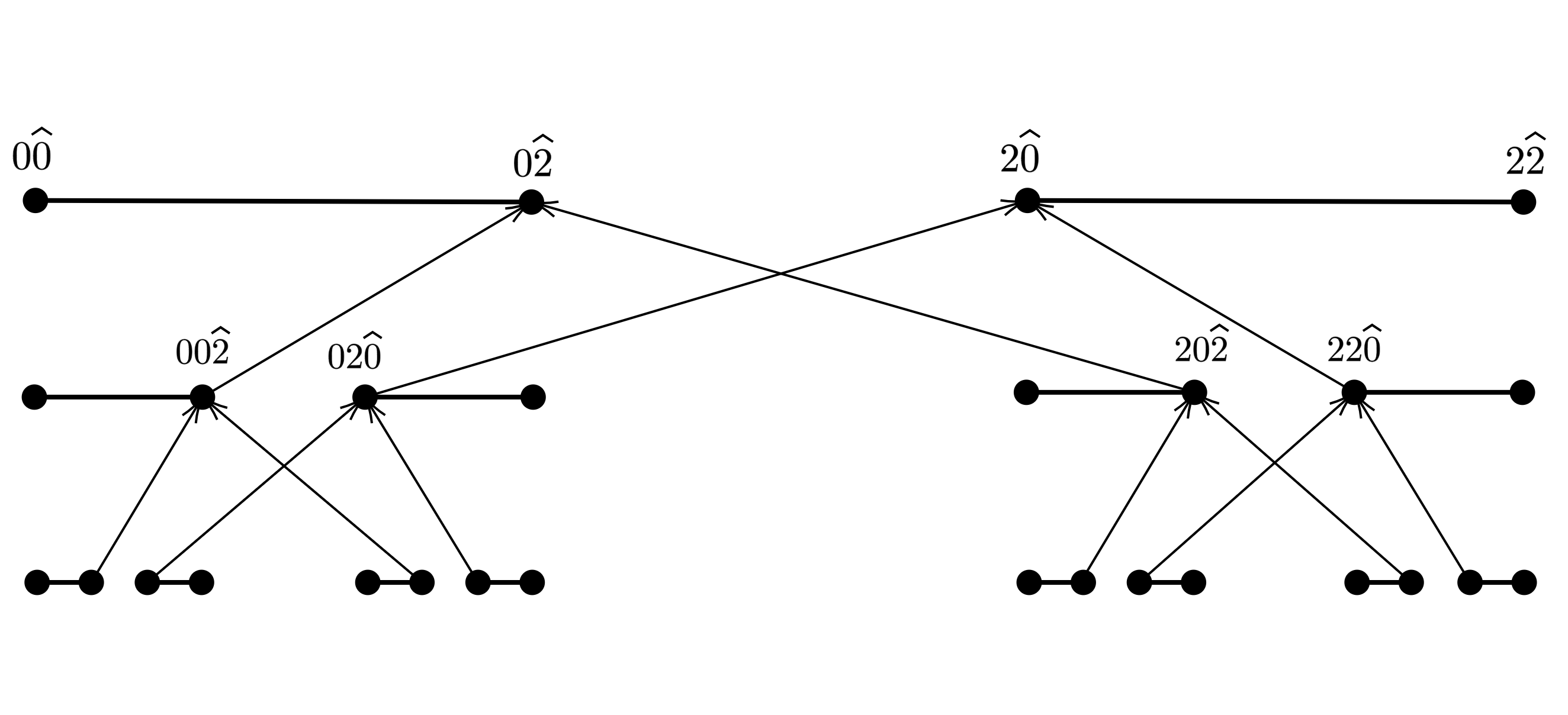}
\caption{Ternary representation of the pairs $a_n$ and $b_n$, and the action of $\sigma$.}
\label{fig:free-repres}
\end{center}
\end{figure}

By the result of Godard---see \cite{Godard}---it is known that $\free{\cantor}\equiv \ell_1$. 
Indeed, the linear operator $\phi \colon \ell_1 \to \free{\cantor}$ defined by $\phi(e_n):= m_{b_n,a_n}$ is a linear onto isometry. In particular, for every $t\in \cantor$, we have 

\[
\dd_t= \sum \{d(b_n,a_n) m_{b_n,a_n}: n \text{ such that } b_n\leq t\} 
\]

and therefore

\[ \phi\Big(\sum \{d(b_n,a_n) e_n: n \text{ such that } b_n \leq t\}\Big)= \dd_t. \]


Then, the backward shift $\sigma$ defines an operator $T_\sigma \colon \free{\cantor} \to \free{\cantor}$, so we may consider its conjugated operator in $\ell_1$ through the isometry $\phi$, that is, $ S_\sigma \colon \ell_1 \to \ell_1$ as $ S_\sigma:= \phi^{-1} \circ T_\sigma \circ \phi $. Let us study its action on vectors of $\ell_1$.

We denote by $[r]$ the integer part of a given number $r$. Notice that, due to the chosen enumeration on the substracted intervals, the action of the backward shift is
 \[\sigma(b_n)= b_{[\frac{n}{2}]}, \ \sigma(a_n)= a_{[\frac{n}{2}]}, \ \text{ for } n\geq 2,\]
 while $\sigma(\frac{1}{3})=1$ and $\sigma(\frac{2}{3})=0$.

Hence for every $n\geq 2$, 
\begin{align*}
    T_\sigma (m_{b_n,a_n}) &=
    \dfrac{\dd_{\sigma(b_n)}-\dd_{\sigma(a_n)}}{d(b_n,a_n)} \\
    &=\dfrac{\dd_{b_{[\frac{n}{2}]}}-\dd_{a_{[\frac{n}{2}]}}}{d(b_n,a_n)} \\
    &= \dfrac{d(b_{[\frac{n}{2}]},a_{[\frac{n}{2}]}) }{d(b_n,a_n)} m_{b_{[\frac{n}{2}]},a_{[\frac{n}{2}]}} \\
    &= 3 m_{b_{[\frac{n}{2}]},a_{[\frac{n}{2}]}}
\end{align*}
  and $T_\sigma (m_{b_1,a_1})=T_\sigma (m_{\frac{2}{3},\frac{1}{3}})= - 3 m_{1,0} = - 3 \sum_{n=1}^\infty d(b_n,a_n) m_{b_n,a_n} $. As $0$ and $1$ are fixed points for $\sigma$, it is deduced that $m_{1,0}$ is a fixed point for $T_\sigma$. 

Notice that it is coherent, as it can be checked by a straightforward computation that $S_\sigma (e_1)= -3 \sum_{n=1}^\infty d(b_n,a_n) e_n \in \ell_1$ is a fixed point for $S_\sigma$. Indeed, 
\begin{align*}
    S_\sigma( -3 \sum_{n=1}^\infty d(b_n,a_n) e_n) &= -3\Big( \frac{1}{3}S(e_1) + S( \sum_{n=2}^\infty d(b_n,a_n) e_n) \Big)\\
    &= -3\Big( - \sum_{n=1}^\infty d(b_n,a_n)e_n + \sum_{n=2}^\infty d(b_n,a_n) 3 e_{[\frac{n}{2}]}   \Big)\\
    &= -3\Big( - \sum_{n=1}^\infty d(b_n,a_n)e_n + 2\sum_{n=1}^\infty d(b_n,a_n)  e_n   \Big)\\
    &= -3 \sum_{n=1}^\infty d(b_n,a_n) e_n.
\end{align*}

Then, we conclude the following:

\begin{proposition}
    The operator $S_\sigma \colon \ell_1 \to \ell_1$ is 
    such that $S_\sigma (e_n)= 3 e_{[n/2]}$ for $n\geq 2$, and $S_\sigma (e_1)= -3 \sum_{n=1}^\infty d(b_n,a_n) e_n$ (which, moreover, is a fixed point) is conjugated to $T_\sigma$. Alternatively, the action of $S_\sigma$ on $\ell_1$ sequences can be described as the product by the infinite matrix, 
    \begin{equation*}
M_\sigma= 3
 \begin{bmatrix}
  - d_1 & 1 & 1& 0 & 0 & 0&0&\cdots & \\
   -d_2 & 0 & 0& 1 & 1 &0 &0&\cdots & \\
    -d_3 & 0 & 0& 0 & 0 &1 & 1&\cdots & \\
      \vdots &  \vdots & &  &  \ddots &\ddots & & &
 \end{bmatrix}
\end{equation*}
where $\{d_n\}_{n=1}^\infty =\Big(\frac{1}{3},\frac{1}{3^2},\frac{1}{3^2}, \frac{1}{3^3}, \frac{1}{3^3},\frac{1}{3^3},\frac{1}{3^3},... \Big)$ is the sequence in Equation \ref{eq:Cantor-sequence}.
\end{proposition}

    As the backward shift $\sigma:\cantor \to \cantor$ is a Lipschitz function which is mixing and Devaney Chaotic, we already know that both $T_\sigma$ and its conjugated operator $S_\sigma$ also inherit these properties, which was already stated in \cite{ACP2021}. Below, we show--as a straightforward application of the Lipschitz d-Hypercyclicity Criterion (Theorem \ref{thm:lip-d-hcc})--that any finite collection of different powers of $T_\sigma\in \mathcal{L}(\free{\cantor}) $ (and so, $S_\sigma \in \mathcal{L}(\ell_1)$) is disjoint mixing.

\begin{proposition} \label{prop:backward-cantor}
    Let $\sigma: \cantor \to \cantor$ be the backward shift map in the usual ternary Cantor set. Therefore $T_{\sigma^{m_1}},...,T_{\sigma^{m_N}}$ are disjoint mixing operators in $\free{\cantor}$ whenever $m_i \neq m_j$ for every $i\neq j$, with $i,j\in \{1,...,N\}$.
\end{proposition} 

\begin{proof}
    It is enough to see that the family of maps $f_i:= \sigma^{m_i}$ for $i=1,...,N$ satisfy the Lipschitz d-Hypercyclicity Criterion (Theorem \ref{thm:lip-d-hcc}) for any strictly increasing sequence $\{n_{k}\}_{k=1}^\infty \subset \NN$. Take as $M_0=M_1= \cdots = M_N$ the subset of numbers of the Cantor set $\cantor$ having finite ternary representation. Now, for the ``inverse maps'', for every $i=1,...,N$ and $k\in \NN$ we will consider the mappings $g_{i,k}\colon \cantor \to \cantor$ defined as
    \[g_{i,k}= (w^{m_i})^k,\]
    where $w:\cantor \to \cantor$ is the usual ``forward shift'' on the Cantor set. Thus, for every element $x\in \cantor$ with finite ternary representation, for $i=1,...,N$, we have that 
  \begin{align*}
  (i)& \   f_i^{n_k}(x) \xrightarrow{k\to \infty} 0\\
  (ii)&\ g_{i,k}(x)\xrightarrow{k\to \infty}  0.
  \end{align*}
  Now, notice that for $m,l\in \NN$
  \begin{equation*}
    \sigma^m\circ w^l= 
      \begin{cases}
        \sigma^{m-l} & \text{ if } m>l;\\
      \id & \text{ if } m=l;\\
      w^{l-m} & \text{ if } l>m,\\
      \end{cases}
  \end{equation*}
 so it follows that for every $x\in \cantor$ with finite ternary representation, and for $i,j\in \{1,...,N\}$, we have that
  \[ (iii) \ (f_{i}^{n_k} \circ g_{j,k}) (x) \xrightarrow{k\to \infty } \Delta_{i,j} x\]
  (where, recall, $\Delta_{i,j} x= x$ when $i=j$ and $\Delta_{i,j}x=0$ when $i\neq j$).

  Then, $T_{\sigma^{m_1}},T_{\sigma^{m_2}},...,T_{\sigma^{m_N}}$ are disjoint mixing operators in $\free{\cantor}$, as we wanted to prove.
\end{proof}

\begin{corollary}
    Any finite collection of different powers of the aforementioned operator $S_\sigma \in \mathcal{L}(\ell_1)$ is disjoint mixing. 
\end{corollary}



\begin{remark}
    \rm
Although we preferred to illustrate the usage of Theorem \ref{thm:lip-d-hcc}, an alternative--and equally easy--way of proving Proposition \ref{prop:backward-cantor} would be showing that any finite collection of powers of the backward shift map acting on $\cantor$ is indeed disjoint mixing, and then concluding that this property is inherited to the corresponding family of operators in $\free{\cantor}$ by Remark \ref{rem:mixing} and Corollary \ref{cor:lip-d-hyp}. 

\eor
\end{remark}

\subsection{The anti-symmetric tent map}\label{subsec:anti-sym-tent-map}

First, it might be useful to consider an auxiliary family of functions. Let $Z_p\colon [0,1] \to [0,1]$ to be the ``$p$-th zig-zag map'', defined by 
 \begin{equation*}
   Z_p(x)= 
      \begin{cases}
       px-2k  & \text{ if } x\in [\frac{2k}{p},\frac{2k+1}{p}], \ 0\leq k \leq \frac{p-1}{2};\\
      -px+2k+2 & \text{ if }  x\in [\frac{2k+1}{p},\frac{2k+2}{p}], \ 0\leq k \leq \frac{p-2}{2}.\\
      \end{cases}
  \end{equation*}

Notice that the function $Z_p$ has $p$ lines with a slope alternating between $p$ and $-p$, so $Z_p$ is $p$-Lipschitz. In particular, $Z_2$ is the usual tent map on the interval $[0,1]$.

 \begin{figure}[htb!]
\begin{center}
\includegraphics[width=5in]{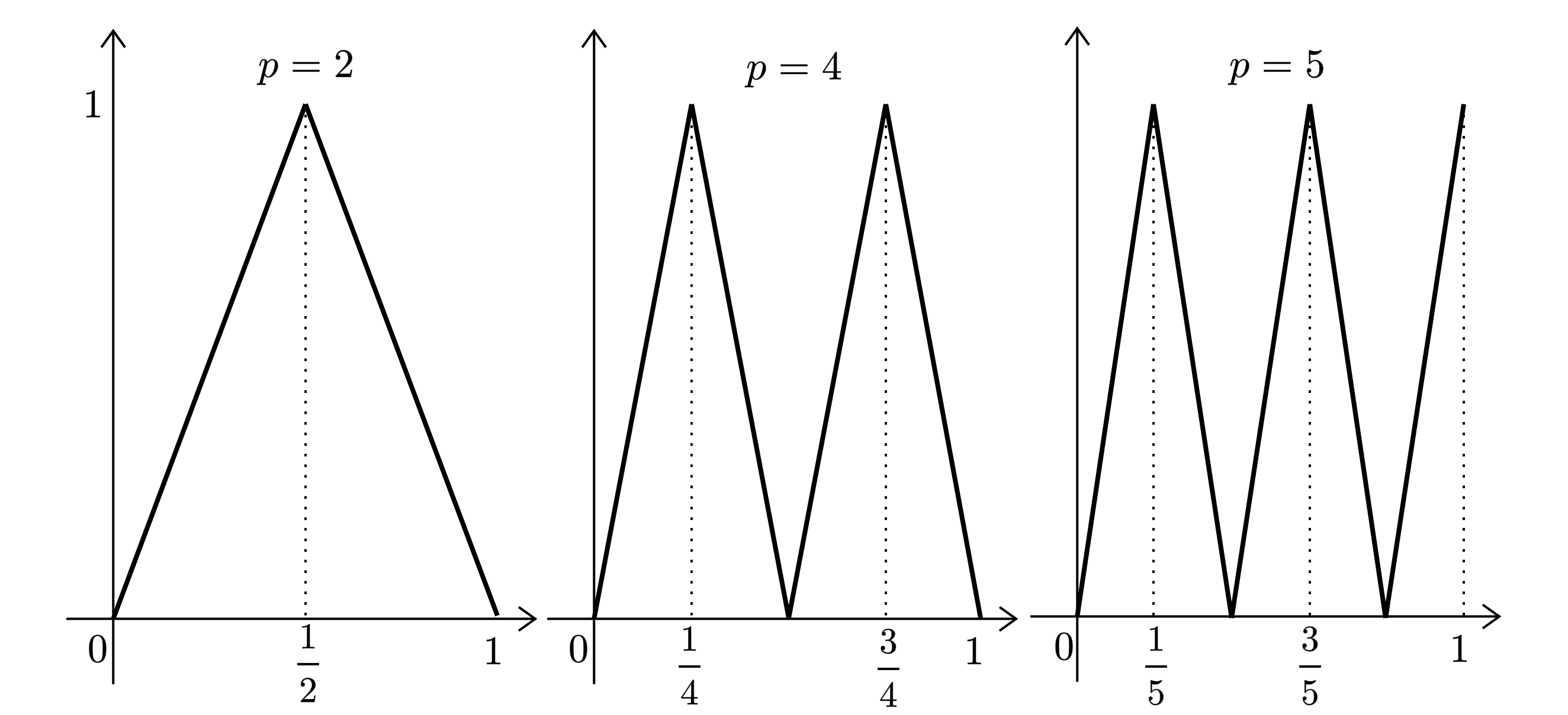}
\caption{Representation of the functions $Z_p$.}
\label{fig:zigzag}
\end{center}
\end{figure}

Now, let $M=[-1,1]$, where $0$ is the distinguished point. Consider $f\colon M \to M$ to be the \textit{anti-symmetric tent map}, defined as

\begin{equation*}
   f(x)= 
      \begin{cases}
       Z_2(x)  & \text{ if } x\in [0,1];\\
      -Z_2(-x) & \text{ if }  x \in [-1,0].\\
      \end{cases}
  \end{equation*}

  Notice that, for every $j\in \NN$,
\begin{equation*}
   f^j(x)= 
      \begin{cases}
       Z_2^j(x)= Z_{2^j}(x)  & \text{ if } x\in [0,1];\\
      -Z_2^j(x)=-Z_{2^j}(-x) & \text{ if }  x \in [-1,0].\\
      \end{cases}
  \end{equation*}

Thus, $f^m$ is a $2^m$-Lipschitz mapping on $M$ with $0$ being a fixed point. As $f([0,1])=[0,1]$ and $f([-1,0])=[-1,0]$, it is known that $f$ is not transitive---although $T_f$ is mixing and chaotic, as pointed out in \cite{ACP2021}. Below, we prove that despite no power of $f$ is even transitive, any finite collection of different Lipschitz-free operators of the form $T_{f^{m}}$ where $m\in \NN$ is a family of disjoint mixing operators in $\free{M}$. 

\begin{proposition}\label{prop:anti-sym-tent-map}
    Let $M=[-1,1]$ and $f\colon M\to M$ the anti-symmetric tent map in $M$. Therefore, $T_{f^{m_1}},...,T_{f^{m_N}}$ are disjoint mixing operators in $\free{M}$ whenever $m_i \neq m_j$ for every $i\neq j$ with $i,j\in \{1,...,N\}$.
\end{proposition}
\begin{proof}
Once again, it is enough to see that the functions $f^{m_1},...,f^{m_N}$ satisfy the Lipschitz d-Hypercyclicity Criterion for any strictly increasing sequence $\{n_k\}_{k=1}^\infty$.
First, take as $M_0=M_1= \cdots =M_N$ the subset of dyadic rational numbers in $M$, i.e., $x_0\in M_0$ if $x_0 \in M$ is of the form $x_0= \pm \frac{r}{2^q}$. Notice that every element of $M_0$ is eventually $0$ after a finite number of iterations, as if $x_0=\pm \frac{r}{2^q}$, then $f^q(x_0)=\pm 1$, hence $f^m(x_0)=0$ for every $m> q$.

     Thus, for every $x_0 \in M_0$ and $i=1,...,N$
\begin{align*}
 (i) \   (f^{m_i})^{n_k}(x_0) & \xrightarrow{k\to \infty } 0.
\end{align*}

Now, for the sequence of ``inverse mappings'', for $i=1,...,N$ and for every $k\in \NN$, consider
\[g_{i,k}:= (g^{m_i})^{n_k},\]
 where $g\colon M \to M $, $g(x):= \frac{x}{2}$. Notice that, as $g^m(x)=\frac{x}{2^m}$, then for every $x\in M$ and $i=1,...,N$
\begin{align*}
  (ii) \   g_{i,k}(x) & \xrightarrow{k\to \infty } 0.
\end{align*}

Finally, for the step $(iii)$, divide the argument in 3 subcases:

-First, notice that $f \circ g= \id$, so it is clear that for every $i=1,...,N$, we have that $(f^{m_i})^{n_k}\circ g_{i,k}= \id$.

-For the second subcase, if $m<l$ (then, the ``compression factor'' of $g^l$ is greater than the ``expansive factor'' of $f^m$), we have that 
\[(f^m \circ g^l)([-1,1])\subset [-\frac{1}{2^{l-m}},\frac{1}{2^{l-m}}].\]
Thus, if $i\neq j$ and $m_i < m_j$, we have
\[((f^{m_i})^{n_k} \circ g_{j,k})(M)\subset [-\frac{1}{2^{(m_j-m_i) n_k}},\frac{1}{2^{(m_j-m_i)n_k}}],\]
so for every $x\in M$,
\[((f^{m_i})^{n_k} \circ g_{j,k})(x) \xrightarrow{k\to \infty} 0. \]

-For the third subcase, if $m>l$, then $f^m\circ g^l=f^{m-l}$, by taking any dyadic number $x_0\in M_0$, it follows that if $i\neq j$ and $m_i > m_j$, then
\[((f^{m_i})^{n_k} \circ g_{j,k})(x_0)= (f^{m_i-m_j})^{n_k}(x_0),\]
which is eventually $0$ for a big enough $k\in \NN$.

 Indeed, for $m,l\in \NN$
  \begin{equation*}
    f^m\circ g^l= 
      \begin{cases}
        f^{m-l} & \text{ if } m>l;\\
      \id & \text{ if } m=l;\\
      g^{l-m} & \text{ if } l>m,\\
      \end{cases}
  \end{equation*}

As a conclusion of the discussion above, we reach that, for $i,j\in \{1,...,N\}$ and $x_0 \in M_0$
  \[ (iii) \ (f_{i}^{n_k} \circ g_{j,k}) (x_0) \xrightarrow{k\to \infty } \Delta_{i,j} x_0\]
  (where, recall, $\Delta_{i,j} x= x$ when $i=j$ and $\Delta_{i,j}x=0$ when $i\neq j$). So, applying Theorem \ref{thm:lip-d-hcc}, the proof is over.

\end{proof}

\begin{remark}
    Notice that in both examples, either with the backward shift on the ternary Cantor set or the anti-symmetric tent map, its associated Lipschitz-free operator $T_f$ has the property that any finite collection of its different powers is disjoint mixing (in particular, disjoint transitive). This implies that, in both cases, $T_f$ is $\mathcal{AP}$-hypercyclic (see \cite{CaMu22}).
    \eor
\end{remark}

 \section*{Acknowledgements}
We would like to thank R. J. Aliaga for interesting discussions around the example in Subsection \ref{subsec:shift-cantor}. The authors were supported by  Generalitat Valenciana (through Project PROMETEU/2021/070) and by MICIU/AEI/
10.13039/501100011033 and by ERDF/EU (through Projects PID2019-105011GB-I00 and PID2022-139449NB-I00). The first author was also supported by Generalitat Valenciana (through the predoctoral fellowship CIACIF/2021/378) and by MICIU/AEI/10.13039/
501100011033 and by ERDF/EU (through Project PID2021-122126NB-C33).

\printbibliography
\end{document}